\documentclass{article}
\usepackage{amsthm,amsfonts,amsbsy,amssymb,amsmath,graphicx}
\usepackage{mathrsfs}
\usepackage[all]{xy}

\DeclareMathAlphabet{\mathdutchcal}{U}{dutchcal}{m}{n}

\newtheorem{theorem}{Theorem}
\newtheorem*{theorem*}{Theorem}

\newtheorem{corollary}{Corollary}

\newtheorem{proposition}{Proposition}
\theoremstyle{definition}
\newtheorem{definition}{Definition}
\theoremstyle{remark}
\newtheorem{remark}{Remark}
\newtheorem{example}{Example}

\def\Z{{\mathbb Z}}

\def\R{{\mathbb R}}

\newcommand{\id}{\mathrm{id}}
\DeclareMathOperator{\rot}{rot}
\DeclareMathOperator{\sign}{sign}

\title{Homotopy index polynomials for knotoids}
\author{Igor Nikonov}
\date{}

\begin{document}

\maketitle

\begin{abstract}
We define homotopy index polynomials and a homotopy type of knotoids in an oriented surface, and consider some basic properties of these polynomials. We show that for planar knotoids the homotopy polynomials are not weaker than previously defined index polynomials, and that the homotopy index polynomials detect non-rotatability of spherical knotoids.
\end{abstract}

\section{Introduction}

The notion of knotoids was introduced by Turaev~\cite{Turaev} as a natural extension of classical knot theory by considering open-ended diagrams. A \emph{knotoid diagram} $D$ is an oriented surface $S$ is a generic immersion of the interval $[0,1]$ in the interior of $S$ whose only singularities are transversal double points endowed
with over/undercrossing data. The images of $0$ and $1$ are called the \emph{tail} and the \emph{head} of a knotoid. Two diagrams are considered \emph{equivalent} if one can get one from the other by a finite sequence of Reidemeister moves (Fig.~\ref{fig:reidmoves}) and isotopies of the surface. An (oriented) \emph{knotoid} is an equivalence class of such diagrams. A knotoid in the plane $\mathbb R^2$ is called \emph{planar}, and a knotoid in the sphere $S^2$ is called \emph{spherical}.

\begin{figure}[h]
    \centering
    \includegraphics[width=0.5\linewidth]{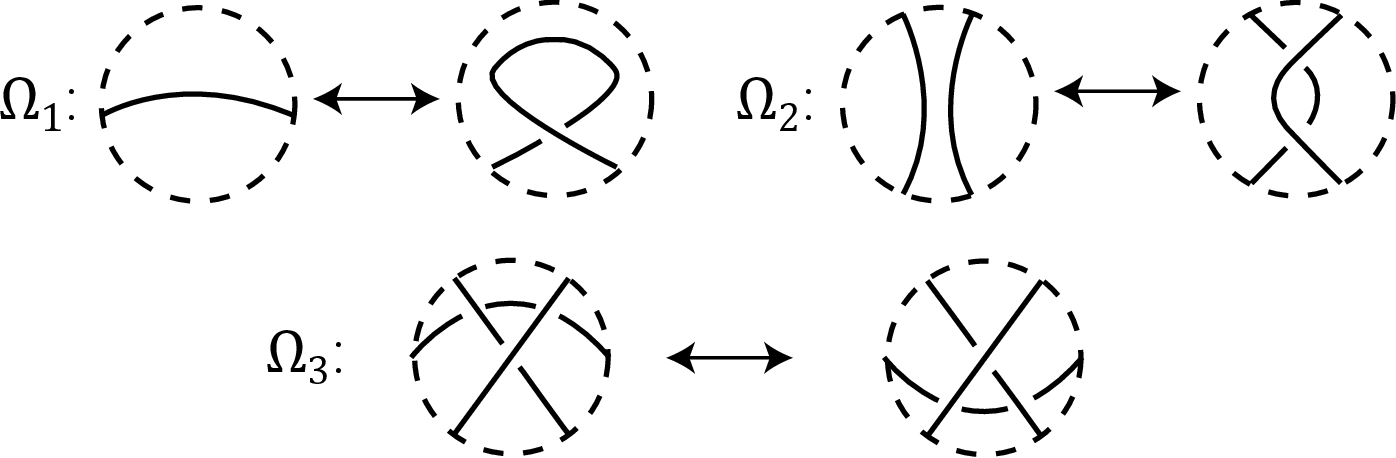}
    \caption{Reidemeister moves}
    \label{fig:reidmoves}
\end{figure}

Note that moves on knotoids must not involve the end points (Fig.~\ref{fig:forbidden_moves}).

\begin{figure}[h]
    \centering
    \includegraphics[width=0.5\linewidth]{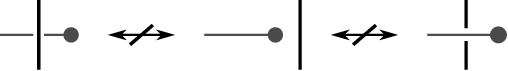}
    \caption{Forbidden moves for knotoids}
    \label{fig:forbidden_moves}
\end{figure}

The construction of knotoid invariants have become a central topic in knotoid theory. Among them we distinguish index polynomials~\cite{CAL26,FL24,FLV25,GK17,KIL18} which has the form
\[
I(D)=\sum_{c\in \mathcal C(D)} \sign(c) \iota(c).
\]
where $\mathcal C(D)$ is the set of crossings of the diagram $D$, $\sign(c)$ is the sign of the crossing (Fig.~\ref{fig:crossing_sign}), and $\iota(c)$ is an index of the crossing $c$ valued in an abelian group.
\begin{figure}[h]
    \centering
    \includegraphics[width=0.2\linewidth]{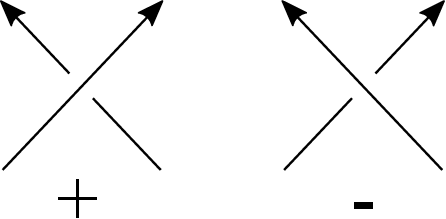}
    \caption{Sign of a crossing}
    \label{fig:crossing_sign}
\end{figure}
An \emph{index} here means a way to assign labels to the crossings of diagrams so that
\begin{itemize}
    \item the index of a crossing remains unchanged under Reidemeister moves;
    \item two crossings which participate in a second Reidemeister move, have the same index value.
\end{itemize}

Indices in different forms have been successfully used in the virtual knot theory~\cite{C,Cheng,F,FK,Henrich,ILL,J,K2,KPV,ST,T}. The notion of index was formalized by Z. Cheng~\cite{Cheng21}. 
A special form of index is parity defined by V.O. Manturov~\cite{Mparity} and later developed in a series of papers~\cite{IMN11,IMN14,IMN15}.
Reviews of index and parity can be found in~\cite{CGX,CFGMX}, see also~\cite{Nind}.

The goal of this paper to define another index polynomial. Following~\cite{Ntribe}, we consider two indices of crossings of knotoid diagrams:
\begin{itemize}
    \item \emph{order index} $o$ that splits the crossings into early undercrossings and early overcrossings (Fig.~\ref{fig:early_underovercrossing}). S. Kim, I.H. Im, and S. Lee~\cite{KIL18} used this splitting to define a two-variable index polynomial.

\begin{figure}[h]
    \centering
    \includegraphics[width=0.5\linewidth]{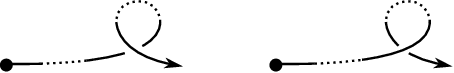}
    \caption{An early undercrossing (left) and an early overcrossing (right)}\label{fig:early_underovercrossing}
\end{figure}

    \item \emph{homotopy index} $h$ that is the homotopy type of the based half of a crossing (Fig.~\ref{fig:based_half}). 

\begin{figure}[h]
    \centering
    \includegraphics[width=0.25\linewidth]{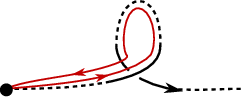}
    \caption{Based half of a crossing}
    \label{fig:based_half}
\end{figure}
\end{itemize}

Using these indices, we define the \emph{early undercrossing homotopy polynomial} $H_u$ and the \emph{early overcrossing homotopy polynomial} $H_o$ of a knotoid diagram $D$:
\[
H_u(D)=\sum_{c\in\mathcal C_u(D)}\sign(c)(h(c)-1),\quad 
H_o(D)=\sum_{c\in\mathcal C_o(D)}\sign(c)(h(c)-1).
\]
Here $\mathcal C_u(D)$ (resp., $\mathcal C_o(D)$) is the set of early undercrossings (resp., early overcrossings) of the diagram $D$. The polynomials take values in the free module generated by the set of orbits of the fundamental group of the surface with two punctures by an action of the braid group on two strands.

The paper is organized as follows. In Section~\ref{sec:definition} we define the homotopy polynomials and homotopy type of knotoids in an oriented surface, and consider some basic properties of the polynomials. In Section~\ref{sec:planar_knotoids} we consider the case of planar knotoids, demonstrate how the homotopy polynomials are computed, and show that they are not weaker than previously defined index polynomials. Section~\ref{sec:spherical_knotoids} is devoted to application of the invariant to the spherical knotoid. In particular, we show that homotopy index polynomials detect non-rotatability.


\section{Homotopy index polynomials}\label{sec:definition}

\subsection{Index}

Let us define an index as in the paper~\cite{Nif}.

\begin{definition}
    An \emph{index} $\iota$ with coefficients in a set $I$ is a family of maps $\iota_D\colon\mathcal C(D)\to I$, $D$ is a knotoid diagram, such that
\begin{itemize}
\item[(I0)] for any diagrams $D$ and $D'$ connected by a Reidemeister move, and any crossing $c\in\mathcal C(D)$ that survives the move, one has $\iota_{D}(c)=\iota_{D'}(c)$;
\item[(I2)] $\iota_D(c_1)=\iota_D(c_2)$ for any crossings $c_1,c_2\in\mathcal C(D)$ to which a decreasing second Reidemeister move can be applied.
\end{itemize}

We suppose that crossings participating in a third Reidemeister move survive since a correspondence can be set up between them as shown in Fig.~\ref{fig:R3crossings}.

\begin{figure}[h]
    \centering
    \includegraphics[width=0.4\linewidth]{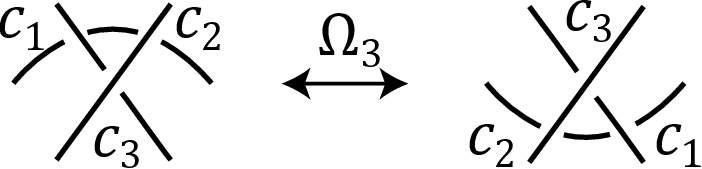}
    \caption{Crossings in a third Reidemeister move}
    \label{fig:R3crossings}
\end{figure}

An index $\iota$ with coefficients in an abelian group $A$ is called \emph{$\Omega_1$-reduced} if
\begin{itemize}
\item[(I1)] $\iota_D(c)=0$ for any crossing $c\in\mathcal C(D)$ to which a decreasing first Reidemeister move can be applied.
\end{itemize}
\end{definition}

\begin{example}\label{exa:index}
    1. Define the \emph{order index} $o$ of a crossing $c$ to be $o(c)=-1$ if $c$ is an early undercrossing, and $o(c)=+1$ if $c$ is an early overcrossing (Fig.~\ref{fig:early_underovercrossing}). One can check that $o$ satisfies conditions (I0) and (I2).

    2. For a crossing $c$ of a knotoid diagram, the oriented smoothing at $c$ splits $D$ into the \emph{left half} $D^l_c$ and the \emph{right half} $D^r_c$  (Fig.~\ref{pic:knotoid_halves}). We define also the \emph{signed halves} $D^\pm_c$ by the formula $D^{sgn(c)}_c=D^l_c$ and $D^{-sgn(c)}_c=D^r_c$.

\begin{figure}[h]
\centering\includegraphics[width=0.5\textwidth]{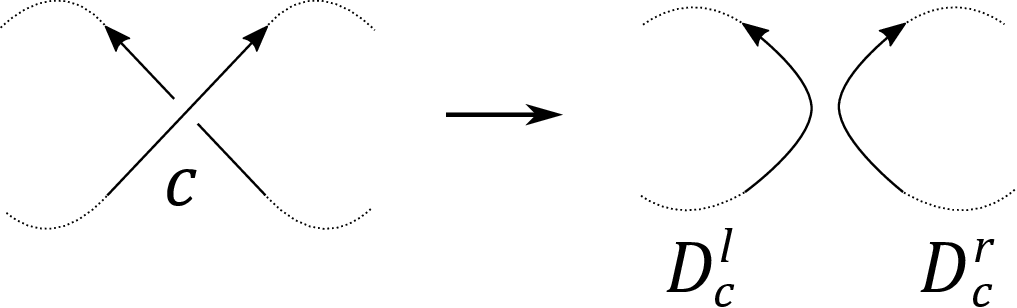}
\caption{The left and the right halves of the diagram}\label{pic:knotoid_halves}
\end{figure}

For any knotoid diagram $D$ and any crossing $c\in\mathcal C(D)$, define its index as the intersection number $Ind_D(v)=D\cdot D^+_v\in\Z$. One can check that it is an $\Omega_1$-reduced index. The index $Ind_D(c)=w_D(c)=-o_D(c)\cdot i(c)$ where $w_D(c)$ is the weight of the affine index polynomial~\cite{GK17}, and $i(c)$ is the intersection index from~\cite{KIL18}. The maps $w_D(c)$ and $i(c)$ are also $\Omega_1$-reduced indices.

Note that the expression $\iota_D(c)=t^{Ind_D(c)}-1$ is an $\Omega_1$-reduced index with coefficients in $\Z[t,t^{-1}]$.

\end{example}

The following statement follows immediately from the definition of index (cf.~\cite[Section 9.2]{Nif}).

\begin{proposition}\label{prop:index_polynomial}
    Let $\iota$ be an $\Omega_1$-reduced index with coefficients in an abelian group $A$. Then 
\[
lk_\iota(D)=\sum_{c\in\mathcal C(D)}\sign(c)\cdot\iota_D(c)\in A
\]
is an invariant of oriented knotoids.
\end{proposition}

The invariant $lk_\iota$ is called the \emph{index polynomial} (or the \emph{linking invariant}) of the index $\iota$.

\begin{example}\label{exa:linking_invariant}
  1.  The index polynomial of the index $t^{Ind_D(c)}-1$ from Example~\ref{exa:index} is the affine index polynomial~\cite{GK17}
\[
P_D(t)=\sum_{c\in\mathcal C(D)}\sign(c)(t^{Ind_D(c)}-1)=\sum_{c\in\mathcal C(D)}\sign(c)(t^{w_D(c)}-1).
\]
\end{example}

In order to define the homotopical index for knotoids in a surface, we need to recall the pure framed braid group and its action on the fundamental group.

\subsection{Pure framed braid group}

For a connected oriented surface $S$ and a subset $X\subset S$, let $Diff^+(S,X)$ denote the set of diffeomorphisms $\psi$ of $S$ such that $\psi$ preserves the orientation and $\psi|_{\partial S\cup X}=\id_{\partial S\cup X}$. Denote $Diff^+(S)=Diff^+(S,\emptyset)$, and let $Diff^+_0(S,X)$ be the connected component of $\id_S$ in $Diff^+(S,X)$.

Let $\mathbb D^2=\{z\in\mathbb C\mid |z|\le 1\}$ be the standard disk. Fix an embedding $U$ of two copies of $\mathbb D^2$ into $S$. Let disjoint disks $U_0, U_1\subset S$ be the image of $U$, and $z_i\in \partial U_i$, $i=0,1$, the images of $1\in\mathbb D^2$.

The homotopy group $FP_2(S)=\pi_1(Emb(\mathbb D^2\sqcup \mathbb D^2,S),U)$ is called the \emph{pure framed braid group} on two strands in the surface $S$~\cite{BG12}. The pure framed braid group is an extension of the pure braid group $P_2(S)$ on two strands in $S$:
\[
1\to \Z^2\to FP_2(S)\to P_2(S)\to 1.
\]
The kernel of the projection from $FP_2(S)$ onto $P_2(S)$ is generated by full turns of the disks $U_i$, $i=0,1$. 

The fibration
\[
Diff^+(S,U_0\cup U_1) \to Diff^+(S)\to Emb(\mathbb D^2\sqcup \mathbb D^2,S)
\]
induces the homotopical exact sequence 

\[
\pi_1(Emb(\mathbb D^2\sqcup \mathbb D^2,S),U)\stackrel{\partial}{\rightarrow} \pi_0(Diff^+(S,U_1\cup U_1))\to\pi_0(Diff^+(S)).
\]
The exactness means that for any $f\in Diff^+(S,U_1\cup U_1)\cap Diff^+_0(S)$, its mapping class is the image $\partial\beta$ of a pure framed braid $\beta$. 

Denote $\bar S = \overline{S\setminus (U_0\cup U_1)}$. Then $\pi_0(Diff^+(S,U_1\cup U_1))\simeq\pi_0(Diff^+(\bar S))$. The map $\partial$ from the pure framed braid group $FP_2(S)$ to the mapping class group $\pi_0(Diff^+(\bar S))$ induces an action of $FP_2(S)$ on the homotopy groups $\pi_1(\bar S,z_0)$ and $\pi_1(\bar S,z_0,z_1)$.

We denote $\bar\pi(S,z_0)=\pi_1(\bar S,z_0)/FP_2(S)$ and $\bar\pi(S,z_0,z_1)=\pi_1(\bar S,z_0,z_1)/FP_2(S)$.

\subsection{Based knotoid diagrams}

Let $O_0$ and $O_1$ be the centers of the disks $U_0$ and $U_1$ defined in the previous section. Fix radii $O_iz_i$ in $U_i$, $i=0,1$.

\begin{definition}\label{def:based_diagram}
Given a knotoid diagram $D$ in $S$, we say that $D$ is a \emph{based diagram} if $O_0$ is the tail of $D$, $O_1$ is the head of $D$, and $D\cap U_i=O_iz_i$, $i=1,2$, Fig.~\ref{fig:based_diagram}.

\begin{figure}[h]
    \centering
    \includegraphics[width=0.3\linewidth]{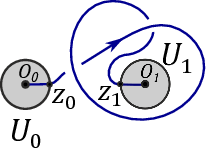}
    \caption{A based knotoid diagram}
    \label{fig:based_diagram}
\end{figure}

Let $D$ be a based knotoid diagram. An isotopy/Reidemester move on $D$ is called \emph{based} if the transformation does not affect the part of $D$ in $U_0$ and $U_1$.

Two based knotoid diagrams are \emph{based equivalent} if they are connected by a sequence of based Reidemeister moves and based isotopies. 
\end{definition}

\begin{remark}\label{rem:based_diagram}
    The map $D\mapsto \bar D=D\cap \bar S$ establishes a bijection between the set of based knotoid diagrams and the set of long knot diagrams with ends $z_0,z_1$ in the surface $\bar S$. By definition, two based diagrams $D_1$ and $D_2$ are based equivalent if and only if the long knot diagrams $\bar D_1$ and $\bar D_2$ define the same long knot.
\end{remark}

\begin{proposition}\label{prop:based_diagram}
1. For any knotoid diagram $D$, there is an isotopy $f\in Diff^+_0(S)$ such that $D'=f(D)$ is a based diagram.

2. For any based diagrams $D'$ and $D''$ obtained from $D$ by isotopies, there is a map $f\in Diff^+(S,U_1\cup U_2)\cap Diff^+_0(S)$ such that $f(D')=D''$.

3. For equivalent knotoid diagrams $D_1$ and $D_2$, there exist based knotoid diagrams $D_1'$ and $D_2'$ such that $D_i'$ is isotopic to $D_i$, and $D'_1$ and $D'_2$ are based equivalent.
\end{proposition}

\begin{proof}
1. The first statement follows form connectivity of $S$.

2. By the first statement $D'=f_1(D)$, $D''=f_2(D)$ for some $f_1,f_2\in Diff^+_0(S)$. Let $g\in Diff^+_0(S)$ be a map such that it extends the isotopy from $\id_{U_0\cup U_1}$ to $(f_2\circ f_1^{-1})|_{U_0\cup U_1}$, and $g|_{D''}=\id_{D''}$. Then $f=g^{-1}\circ f_2\circ f_1^{-1}$ is the required map.

3. Let $D'_1=f_1(D_1)$, $f_1\in Diff^+_0(S)$, be a based diagram isotopic to $D_1$. The equivalence between $D_1$ and $D_2$ can be presented as a composition of an isotopy $g\in Diff^+_0(S)$ and a number of Reidemeisted moves. Consider the based diagram $D'_2=f_1\circ g^{-1}(D_2)$ of $D_2$. Then $D'_1$ and $D'_2$ are connected by a sequence of Reidemeister moves. Since Reidemeister moves do not affect the head and the tail,  we can suppose that $D'_1$ and $D'_2$ are based equivalent.
\end{proof}

\subsection{Homotopy index}

\begin{definition}\label{def:homotopy_index}
Let $c\in\mathcal C(D)$ be a crossing in a based diagram $D$. Let $D^{cl}_c$ be the closed half of the diagram at the crossing $c$, and $\gamma$ the arc of the diagram from $z_0$ to $c$. Consider the based half $\hat D_c=\gamma D^{cl}_c\gamma^{-1}$ of the crossing (Fig.~\ref{fig:based_half}). Its homotopy class
\[
h_D(c)=[\hat D_c]\in\bar\pi(S,z_0)
\]
is called the \emph{homotopy index} of the crossing $c$.
\begin{figure}[h]
    \centering
    \includegraphics[width=0.7\linewidth]{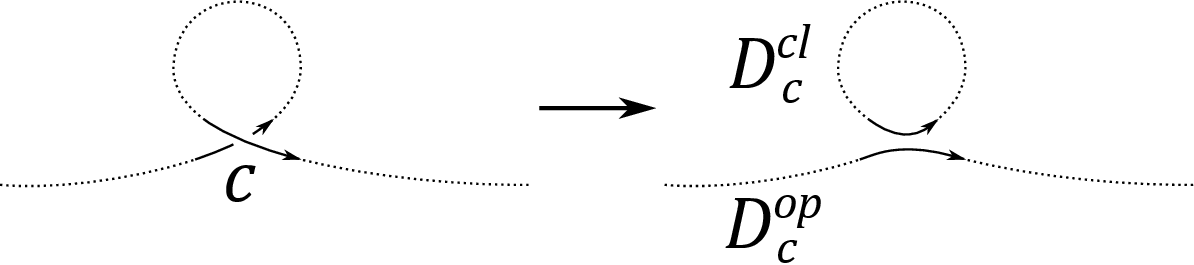}
    \caption{Closed and open halves of a diagram}
    \label{fig:long_knot_halves}
\end{figure}

For a non based knotoid diagram $D$ with an isotopic based diagram $D'=f(D)$, $f\in Diff^+_0(S)$, the homotopy index of a crossing $c\in\mathcal C(D)$ is defined by the formula
\[
h_D(c)=h_{D'}(f(c))
\]
where $f(c)\in\mathcal C(D')$ is the corresponding crossing in $D'$.
\end{definition}

\begin{proposition}\label{prop:homotopy_index}
    The homotopy index $h$ is an index with coefficients in $\bar\pi(S,z_0)$.
\end{proposition}

\begin{proof}
Let us first check that the homotopy index is well defined. Let $D$ be a knotoid diagram and $c\in\mathcal C(D)$. Assume we have two based diagrams $D_i=f_i(D)$, $f_i\in Diff^+_0(S)$, $i=1,2$, isotopic to $D$. Denote $c_i=f_i(c)\in\mathcal C(D_i)$ and $h_i=h_{D_i}(c_i)$, $i=1,2$. We can assume that $g=f_2\circ f_1^{-1}\in Diff^+(S,U_1\cup U_2)$. Then by definition $h_2=g_*(h_1)$. Since $g\in Diff^+_0(S)$, its mapping class is the image of a framed braid $\beta$. Hence, $h_2=\beta(h_1)$, and $h_1$ and $h_2$ present the same element in $\bar\pi(S,z_0)$.

Let $D_1$ and $D_2$ be two diagrams connected by a Reidemeister move and $c\in\mathcal C(D_1)$ a crossing that survives the move. We can suppose that $D_1$ and $D_2$ are based diagrams. Then the based halves of $c$ in $D_1$ and $D_2$ are homotopic. Hence, $h_{D_1}(c)=h_{D_2}(c)$ that implies (I0).

Let $c_1,c_2\in\mathcal C(D)$ participate in a second Reidemeister move. Then their based halves are homotopic (Fig.~\ref{fig:homotopy_type_r2}). Hence, $h_D(c_1)=h_D(c_2)$. 

\begin{figure}[h]
    \centering
    \includegraphics[width=0.5\linewidth]{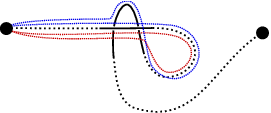}
    \caption{Based halves of crossings participating in $\Omega_2$}
    \label{fig:homotopy_type_r2}
\end{figure}

Thus, $h$ is an index.     
\end{proof}

\subsection{Homotopy type and homotopy index polynomials}

\begin{definition}\label{def:homotopy_type}
Let $D$ be a based diagram of a knotoid $K$ in $S$. The intersection $\bar D=D\cap\bar S$ is a path in $\bar S$ which starts in $z_0$ and ends in $z_1$. The \emph{homotopy type} of the knotoid $K$ is the homotopy class
\[
h(K)=[\bar D]\in\bar\pi(S,z_0,z_1).
\]
\end{definition}

\begin{definition}
    Let $D$ be a diagram of a knotoid $K$ in $S$. The \emph{early undercrossing homotopy index polynomial} of $K$ is 
\[
H_u(K)=\sum_{c\in\mathcal C(D)\colon o(c)=-1}\sign(c)(h_D(c)-1)\in\Z[\bar\pi(S,z_0)],
\]
and the \emph{early overcrossing homotopy index polynomial} of $K$ is 
\[
H_o(K)=\sum_{c\in\mathcal C(D)\colon o(c)=1}\sign(c)(h_D(c)-1)\in\Z[\bar\pi(S,z_0)].
\]
Here $1$ is the trivial homotopy class.
\end{definition}

\begin{theorem}\label{thm:invariance}
1. The homotopy type $h(K)$ is a knotoid invariant.

2. The homotopy index polynomials $H_o(K)$ and $H_u(K)$ are knotoid invariants.
\end{theorem}

\begin{proof}
1. Let $D, D'$ be isotopic based diagrams of $K$. By Proposition~\ref{prop:based_diagram} $D'=f(D)$ for some $f\in Diff^+(S,U_0\cup U_1)\cap Diff^+_0(S)$. Hence, $[\bar D']=\beta[\bar D]$ for some pure framed braid $\beta$. Thus, $[\bar D']$ and $[\bar D]$ define the same element $h(K)$ in $\bar\pi(S,z_0,z_1)$, i.e. $h(K)$ does not depend on the choice of a based diagram in the isotopy class.

Let $D_1$ and $D_2$ are based diagram of $K$. By Proposition~\ref{prop:based_diagram}, there exist based diagrams $D'_1$ and $D'_2$ such that $D'_i$ is isotopic to $D_i$, and $D'_1$ and $D'_2$ are based equivalent. Then $\bar D'_1$ and $\bar D'_2$ are equivalent long knots in $\bar S$. Hence, $[D'_1]=[D'_2]\in\pi_1(\bar S,z_0,z_1)$. Thus, $[D_1]=[D'_1]=[D'_2]=[D_2]$ in $\bar\pi(S,z_0,z_1)$, and $h(K)$ does not depend on the choice of based diagram of $K$.

2. By Proposition~\ref{prop:homotopy_index} and Example~\ref{exa:index}, the expression $\iota_\pm(c)=\frac{1\pm o(c)}{2}(h_D(c)-1)$ is an index. For any crossing $c$ participating in a first Reidemeister move, its closed half and the based half is contractible, hence, $h(c)=1$ and $\iota_\pm(c)=0$. Thus, $\iota_\pm$ is an $\Omega_1$-reduced index with coefficients in $\Z[\bar\pi(S,z_0)]$. By Proposition~\ref{prop:index_polynomial}, $H_o=lk_{\iota_+}$ and $H_u=lk_{\iota_-}$ are knotoid invariants.
\end{proof}



\subsection{Symmetries of knotoids}

\begin{definition}\label{def:knotoid_symmetry}
Given a knotoid $K$, one can assign to it some other knotoids~\cite{GDS19} (Fig.~\ref{fig:knotoid_symmetries}):
\begin{itemize}
    \item the \emph{reversed knotoid} $-K$ is the result of the orientation reversion applied to $K$;
    \item the \emph{mirror knotoid} $\bar K$ is obtained from $K$ by switching all the crossings in a diagram of $K$;
    \item if $K$ is planar or spherical then the \emph{symmetrical knotoid} $K^*$ is defined by applying an orientation changing symmetry of the plane (sphere) to a diagram of $K$;
    \item for a planar (spherical) $K$, its \emph{rotation} $\rot(K)$ is the composition $\rot(K)=\bar K^*$.
\end{itemize}
\end{definition}

\begin{figure}[h]
    \centering
    \includegraphics[width=0.6\linewidth]{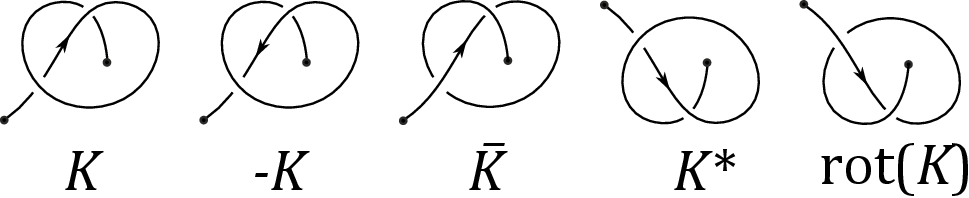}
    \caption{Symmetry operations on a knotoid}
    \label{fig:knotoid_symmetries}
\end{figure}

\begin{proposition}\label{prop:mirror_knotoid}
    For any knotoid $K$, $H_q(\bar K)=-H_{\bar q}(K)$, where $q\in\{o,u\}$, $\bar o=u$, $\bar u=o$.
\end{proposition}
\begin{proof}
    There is a bijection between the crossing sets $\mathcal C(K)$ and $\mathcal C(\bar K)$. For a crossing $c\in\mathcal C(K)$, denote the corresponding crossing in $\bar K$ by $\bar c$. Then for any crossing $c$ we have $\sign(\bar c)=-\sign(c)$, $o(\bar c)=-o(c)$, $h_{\bar K}(\bar c)=h_K(c)$ from which the statement follows. 
\end{proof}

In order to describe the effect of knotoid reversion, we need to extend the homotopy index invariant (cf.~\cite{Next}).

\subsection{Extended homotopy index polynomial}

\begin{definition}\label{def:extended_homotopy_polynomial}
Let $\Pi(S,z_0,z_1)$ be the set of orbits of the set $\Z[\pi_1(\bar S,z_0)]\times\Z[\pi_1(\bar S,z_0)]\times\pi_1(\bar S,z_0,z_1)$ by the diagonal action of the pure framed braid group $FP_2(S)$.

For a based knotoid diagram $D$, define its \emph{extended homotopy index polynomial} as the triple
\begin{gather*}
\tilde H(D)=\left(\sum_{c\in\mathcal C(D)\colon o(c)=1}\sign(c)([\hat D_c]-1), \sum_{c\in\mathcal C(D)\colon o(c)=-1}\sign(c)([\hat D_c]-1), [\bar D]\right)
\end{gather*}
considered as an element of $\Pi(S,z_0,z_1)$.

For a non based knotoid diagram $D$, its extended homotopy index polynomial is defined as the extended polynomial $\tilde H(D')$ of a based diagram $D'$ isotopic to $D$.
\end{definition}

\begin{theorem}\label{thm:extended_polynoial_invariance}
    The extended homotopy index polynomial $\tilde H$ is a knotoid invariants.
\end{theorem}

The proof is analogous to the proof of Theorem~\ref{thm:invariance}.

Note that the natural projections from $\Pi(S,z_0,z_1)$ to $\Z[\bar\pi(S,z_0)]$ and $\bar\pi(S,z_0,z_1)$ map $\tilde H(K)$ to $H_o(K)$, $H_u(K)$ and $h(K)$. Example~\ref{exa:extended_polynomials} below shows that the extended homotopy index polynomial  $\tilde H$ is stronger than the polynomials $H_o$, $H_u$.

Let $\tau\in Diff^+_0(S)$ be a map that swaps the disks $U_0$ and $U_1$. Consider the map $\phi$ of $\Z[\pi_1(\bar S,z_0)]\times\Z[\pi_1(\bar S,z_0)]\times\pi_1(\bar S,z_0,z_1)$ to itself defined by the formula
{\small
\[
\phi\left(\sum_i a_i\alpha_i,\sum_j b_j\beta_j,\gamma\right)=\left(\sum_j b_j\cdot\tau_*(\gamma^{-1}\beta_j^{-1}\gamma), \sum_i a_i\cdot\tau_*(\gamma^{-1}\alpha_i^{-1}\gamma), \tau_*(\gamma^{-1})\right),
\]
}
where $a_i,b_j\in\Z$, $\alpha_i,\beta_j\in\pi_1(\bar S,z_0)$, $\gamma\in\pi_1(\bar S,z_0,z_1)$. This map induces a well defined map $\phi\colon\Pi(S,z_0,z_1)\to\Pi(S,z_0,z_1)$.

\begin{proposition}\label{prop:reversed_knotoid}
    For an oriented knotoid $K$, $\tilde H(-K)=\phi(\tilde H(K))$.
\end{proposition}
\begin{proof}
Let $D$ be a based diagram of $K$, and $-D$ be the reversed diagram. For a crossing $c\in\mathcal C(D)$, denote the corresponding crossing in $\mathcal C(-D)$ by $-c$. Then $\sign(-c)=\sign(c)$, $o(-c)=-o(c)$, and $[\widehat{-D}_{-c}]=[\bar D]^{-1}[\hat D_c]^{-1}[\bar D]\in\pi_1(\bar S,z_1)$. Moreover, $[\overline{-D}]=[\bar D]^{-1}\in\pi_1(\bar S,z_1,z_0)$.

The diagram $-D$ is not based but it is isotopic to a based diagram $\tau(-D)$. Hence,
\begin{multline*}
\tilde H(-K)=\tilde H(\tau(-D))=
\left(\sum_{-c\in\mathcal C(-D)\colon o(-c)=1}\sign(-c)\cdot\tau_*([\widehat{-D}_{-c}]-1), \right.\\
\left. \sum_{-c\in\mathcal C(-D)\colon o(-c)=-1}\sign(-c)\cdot\tau_*([\widehat{-D}_{-c}]-1),\tau_*([\overline{-D}])\right)=\\
\left(\sum_{c\in\mathcal C(D)\colon o(c)=-1}\sign(c)\cdot(\tau_*([\bar D]^{-1}[\hat D_c]^{-1}[\bar D])-1), \right.\\
\left. \sum_{c\in\mathcal C(D)\colon o(c)=1}\sign(c)\cdot(\tau_*([\bar D]^{-1}[\hat D_c]^{-1}[\bar D])-1),\tau_*([\bar D]^{-1})\right)=\phi(\tilde H(K)).
\end{multline*}
\end{proof}

\subsection{Skew homotopy index polynomial}

For a knotoid $K$, we call the difference $\Delta H(K)=H_o(K)-H_u(K)$  the  \emph{skew homotopy index polynomial} of $K$.

\begin{proposition}\label{prop:clasp_move}
Clasp moves (Fig.~\ref{fig:clasp_move}) do not change $\Delta H$.

\begin{figure}[h]
    \centering
    \includegraphics[width=0.2\linewidth]{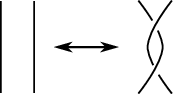}
    \caption{Clasp move}
    \label{fig:clasp_move}
\end{figure}
\end{proposition}
\begin{proof}
    Let $D_1$ and $D_1$ be knotoid diagrams connected by a clasp move. Assume that $\mathcal C(D_2)=\mathcal C(D_1)\cup\{c_1,c_2\}$ where $c_1$ and $c_2$ are the clasp crossings. For any crossing $c\in\mathcal C(D_1)$, $h_{D_2}(c)=h_{D_1}(c)$ and $o_{D_1}(c)=o_{D_2}(c)$. On the other hand, $\sign(c_1)=\sign(c_2)$, $h_{D_2}(c_1)=h_{D_2}(c_2)$, and $o_{D_2}(c_1)=-o_{D_2}(c_2)$. Assume that, $o_{D_2}(c_1)=+1$. Then
\begin{gather*}
    H_o(D_2)=H_o(D_1)+\sign(c_1)(h(c_1)-1),\\
    H_u(D_2)=H_u(D_1)+\sign(c_2)(h(c_2)-1).
\end{gather*}
Hence,
\begin{multline*}
    \Delta H(D_2)=\Delta H(D_1)+\sign(c_1)(h(c_1)-1)-\sign(c_2)(h(c_2)-1)=\Delta H(D_1).
\end{multline*}

\end{proof}


\begin{corollary}\label{cor:deltaH}
    The polynomial $\Delta H(K)$ depends only on the homotopy type of $K$.
\end{corollary}
\begin{proof}
    Let $K_1$ and $K_2$ be knotoid such that $h(K_1)=h(K_2)$, and $D_1$ and $D_2$ based diagrams of $K_1$ and $K_2$. Since $h(K_1)=h(K_2)$ then $[\bar D_1]=f_*[\bar D_2]\in\pi_1(\bar S,z_0,z_1)$ for some $f\in Diff^+(S,U_0\cup U_1)\cap Diff^+_0(S)$. Then $D_2'=f(D_2)$ is a based diagram of $K_2$ such that $[\bar D_1]=[\bar D'_2]$. Below we assume that $D'_2=D_2$.

    Since $\bar D_1$ and $\bar D_2$ are homotopic in $\bar S$ with fixed ends, the diagrams $D_1$ and $D_2$ are connected by a sequence of based isotopies, Reidemeister moves and clasp moves. These moves do not change the skew homotopy index polynomial, hence, $\Delta H(D_1)=\Delta H(D_2)$.
\end{proof}

\begin{corollary}\label{cor:deltaHtrivial}
    Let $K$ be a knotoid whose homotopy type is trivial. Then $H_o(K)=H_u(K)$.
\end{corollary}
\begin{proof}
    Let $K_0$ be the trivial knotoid. It has a diagram without crossings. Then $H_o(K_0)=H_u(K_0)=0$, and $\Delta H(K)=\Delta H(K_0)=0$. 
\end{proof}


\section{Homotopy polynomials for planar knotoids}\label{sec:planar_knotoids}

Let $S=\mathbb R^2$ be the plane. Choose the disks $U_0$ and $U_1$ with the centers $(0,0)$ and $(1,0)$ and the radii $\epsilon$ for a small $\epsilon>0$. Let $z_0=(\epsilon,0)$ and $z_1=(1-\epsilon,0)$. The fundamental group $\pi_1(\bar S,z_0)$ of $\bar S=\R^2\setminus(U_0\cup U_1)$ is a free group $\langle x,y\rangle$ in two generators $x,y$ (Fig.~\ref{fig:homotopy_generators}).
 
\begin{figure}[h]
    \centering
    \includegraphics[width=0.2\linewidth]{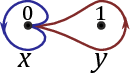}
    \caption{Generators of $\pi_1(\bar S,z_0)$}
    \label{fig:homotopy_generators}
\end{figure}

Let $\delta$ be the segment $z_0z_1$. We identify $\pi_1(\bar S,z_0,z_0)$ with $\pi_1(\bar S,z_0)$ by the formula $\gamma\mapsto \gamma\delta^{-1}$.

The pure braid group $P_2(\R^2)$ is $\Z$, and the pure framed group $FP_2(\R^2)$ is $\Z^3$. The group is generated by Dehn twists $\tau_0,\tau_1,\tau_{01}$ along the curves $\alpha_0,\alpha_1,\alpha_{01}$ in Fig.~\ref{fig:twist_curves}.

\begin{figure}[h]
    \centering
    \includegraphics[width=0.25\linewidth]{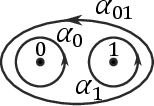}
    \caption{Dehn twist curves}
    \label{fig:twist_curves}
\end{figure}

The generators $\tau_{1}$ and $\tau_{01}$ act on $\pi_1(\bar S,z_0)$ by the identity, and $\tau_0$ acts by conjugation by $x$. On $\pi_1(\bar S,z_0,z_1)$, $\tau_{01}$ is identity, $\tau_0$ multiplies by $x$ from the left, and $\tau_1$ multiplies by $y$ from the right. Hence, 
\[
\bar\pi(S,z_0)=\langle x,y\rangle/Ad(x),\quad \bar\pi(S,z_0,z_1)=x\backslash\langle x,y\rangle/y.
\]
Any element in $\bar\pi(S,z_0)$ has a unique representative which either is equal to $x^k$ or starts with $y^\pm$. Any element in $\bar\pi(S,z_0,z_1)$  has a unique representative which is either $1$ or starts with $y^\pm$ and ends with $x^\pm$.

\begin{example}\label{exa:knotoid_K1}
   Consider the knotoid $K_1$ in Fig.~\ref{fig:planar_knotoid} . Its homotopy type is $h(K_1)=yx$, and the homotopy index polynomials are $H_o(K_1)=yx-1$, $H_u(K_1)=0$. Hence, $K_1$ is not trivial.

\begin{figure}[h]
    \centering
    \includegraphics[width=0.2\linewidth]{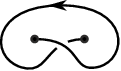}
    \caption{A planar knotoid $K_1$}
    \label{fig:planar_knotoid}
\end{figure}

The extended homotopy index polynomial is $\tilde H(K_1)=(yx-1,0,yx)$.

Note that the Gauss diagram of $K_1$ has one chord, hence, the polynomials which are defined using Gauss diagrams, such as the affine index polynomial $P_K(t)$~\cite{GK17}, the index polynomials $F_D(s,t)$ and $Z^n_D(t)$~\cite{KIL18}, the $F$-polynomial $F_D(u,v)$~\cite{FL24}, the three-value transcendental $H_D(t,y,z)$~\cite{FLV25}, and the coloring-allowed invariants $\mathcal A_D(x_1,\dots,x_n)$~\cite{CAL26}, they all vanish on $K_1$. Thus, the homotopy index polynomials are not weaker than all the aforementioned invariants.  
\end{example}

\begin{example}\label{exa:extended_polynomials}
Consider the knotoids $K_2$ and $K_2'$ in Fig.~\ref{fig:knotoid_for_extended}. Their homotopy index polynomials and homotopy types coincide:
\[
H_o(K_2)=H_u(K_2)=H_o(K'_2)=H_u(K'_2)=y-1,\quad h(K_2)=h(K'_2)=1.
\]
But the extended homotopy index polynomial of $K_2$ is
\[
\tilde H(K_2)=(y-1,y-1,y)=(y-1,y-1,1),
\]
and the extended homotopy index polynomial of $K'_2$ is
\[
\tilde H(K'_2)=(y-1,y-1,x^{-1})=(xyx^{-1}-1,xyx^{-1}-1,1).
\]
Thus, $K_2$ and $K'_2$ are different knotoids, and the extended homotopy index polynomial is stronger than the homotopy index polynomials.

\begin{figure}[h]
    \centering
    \includegraphics[width=0.5\linewidth]{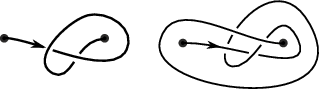}
    \caption{Knotoids $K_2$ (left) and $K_2'$ (right)}
    \label{fig:knotoid_for_extended}
\end{figure}
\end{example}



The plane reflection $\Psi$ relative to the line $z_0z_1$ induces an automorphism $\psi$ of $\pi_1(\bar S,z_0)$ which maps $x$ to $x^{-1}$ and $y$ to $y^{-1}$. 

\begin{proposition}\label{prop:symmetry_knotoid}
  For a planar knotoid $K$, we have $H_q(K^*)=-\psi(H_q(K))$, $H_q(\rot(K))=\psi(H_{\bar q}(K))$, $q\in\{o,u\}$.
\end{proposition}
\begin{proof}
    Let $D$ be a based diagram of $K$. Then $D^*=\Psi(D)$ is a based diagram of $K^*$. For a crossing $c\in\mathcal C(D)$, let $c^*$ denote the corresponding crossing in $\mathcal C(D^*)$. Then $\sign(c^*)=-\sign(c)$, $o(c^*)=o(c)$, $[\widehat{D^*}_{c^*}]=\psi[\hat D_c]$ that implies the first statement.

    The second statement follows from Proposition~\ref{prop:mirror_knotoid}.
\end{proof}

\begin{example}\label{exa:knotoid53}
Consider the knotoid $K_3$ in Fig.~\ref{fig:knotoid53} (knotoid $K7_{53}$ from the table~\cite{GC26}). In order to calculate the homotopy index of crossings, draw cuts which connect the tail and the head with the infinity and label them by $x$ and $y$. A counterclockwise intersection with a cut corresponds to multiplication by the label of the cut, and a clockwise one corresponds to multiplication by the inverse of the label. For example, the homotopy index of $c_3$, i.e. the homotopy class of its based half is $y^{-1}x^{-1}y^{-1}$ (Fig.~\ref{fig:knotoid53_based_half}).  
\begin{figure}[h]
    \centering
    \includegraphics[width=0.3\linewidth]{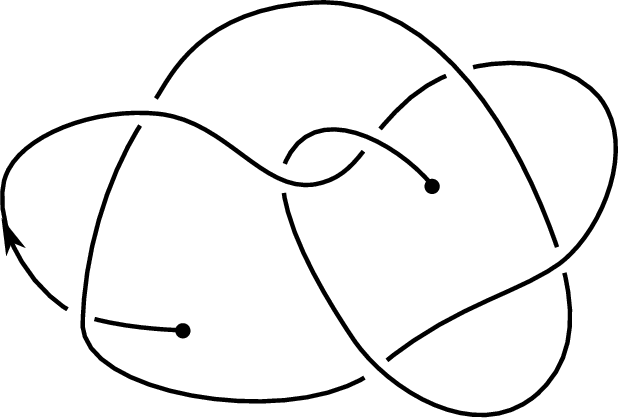}
    \caption{A planar knotoid $K_3$}
    \label{fig:knotoid53}
\end{figure}

\begin{figure}[h]
    \centering
    \includegraphics[width=0.35\linewidth]{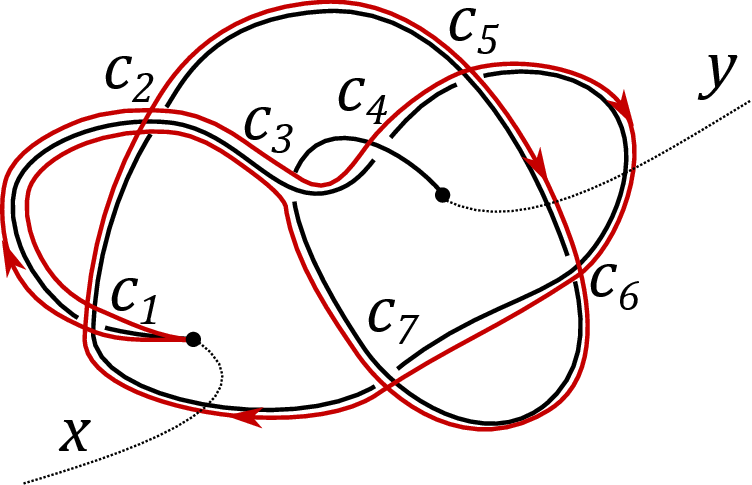}
    \caption{The based half of $c_3$}
    \label{fig:knotoid53_based_half}
\end{figure}

The indices of the crossings are the following
\[
\begin{array}{|c|c|c|c|}

\hline \mbox{crossing} & \mbox{sign} & \mbox{order index} &  \mbox{homotopy index}\\
\hline 
c_1 & +1 & -1 & y^{-1}x^{-1} \\
c_2 & +1 & +1 & y^{-1}x^{-1} \\
c_3 & +1 & +1 & y^{-1}x^{-1}y^{-1} \\
c_4 & +1 & -1 & y^{-1}x^{-1}y^{-1} \\
c_5 & +1 & -1 & y^{-1}x^{-1} \\
c_6 & +1 & +1 & y^{-1}x^{-1} \\
c_7 & +1 & -1 & y^{-1}x^{-1} \\
\hline
\end{array}
\]

Hence, 
\[
H_o(K_3)=2y^{-1}x^{-1}+y^{-1}x^{-1}y^{-1}-3,\quad H_u(K_3)=3y^{-1}x^{-1}+y^{-1}x^{-1}y^{-1}-4.
\]
The homotopy type of the knotoid is $h(K_3)=y^{-1}x^{-1}y^{-1}=y^{-1}x^{-1}\in x\backslash\langle x,y\rangle/y$.

The extended homotopy index polynomial of $K_3$ is
\begin{multline*}
    \tilde H(K_3)=\left(2y^{-1}x^{-1}+y^{-1}x^{-1}y^{-1}-3, 3y^{-1}x^{-1}+y^{-1}x^{-1}y^{-1}-4, y^{-1}x^{-1}y^{-1}\right)=\\
    \left(2y^{-1}x^{-1}+y^{-1}x^{-1}y^{-1}-3, 3y^{-1}x^{-1}+y^{-1}x^{-1}y^{-1}-4, y^{-1}x^{-1}\right)
\end{multline*}

In order to compute the homotopy polynomials of $-K_3$, consider the plane rotation $\tau$ by $\pi$ with the center $(\frac 12,0)$. Is swaps $U_0$ and $U_1$. Denote $\tilde\tau_*(\alpha)=\tau_*(\delta^{-1}\alpha\delta)$.  The map $\tilde\tau_*$ is an automorphism of $\pi_1(\bar S,z_0)$ such that $\tilde\tau_*(x)=y, \tilde\tau_*(y)=x$. For an element $\gamma\in\pi_1(\bar S,z_0,z_1)$, we have 
\[
\tau_*(\gamma^{-1}\alpha\gamma)=\tilde\tau_*(\delta\gamma^{-1}\alpha\gamma\delta^{-1})=\tilde\tau_*((\gamma\delta^{-1})^{-1}\alpha(\gamma\delta^{-1})), \quad \alpha\in\pi_1(\bar S,z_0).
\]
The element $\gamma\delta^{-1}\in\pi_1(\bar S,z_0)$ corresponds to $\gamma$ in the identification of $\pi_1(\bar S,z_0,z_1)$  with $\pi_1(\bar S,z_0)$. By Proposition~\ref{prop:reversed_knotoid} the first component of $\tilde H(-K_3)$ is
\begin{multline*}
    \tau_*\left([\bar K_3]^{-1}\left(3(y^{-1}x^{-1})^{-1}+(y^{-1}x^{-1}y^{-1})^{-1}-4\right)[\bar K_3] \right)=\\
    \tilde\tau_*\left((y^{-1}x^{-1}y^{-1})^{-1}(3xy+yxy-4)(y^{-1}x^{-1}y^{-1}) \right)=\\
    \tilde\tau_*\left(3yx+yxy-4\right)=3xy+xyx-4,
\end{multline*}
and the third component is calculated by the equality
\[
\tau_*(\gamma^{-1})\delta^{-1}=\tau_*(\gamma^{-1}\delta)=\tau_*(\delta^{-1}\delta\gamma^{-1}\delta)=\tilde\tau_*((\gamma\delta^{-1})^{-1}), \quad \gamma\in\pi_1(\bar S,z_0,z_1),
\]
applied to $[\bar K_3]\delta^{-1}=y^{-1}x^{-1}y^{-1}$. Thus, the extended homotopy index polynomial is equal to
\begin{multline*}
    \tilde H(-K_3)=\left(3yx+xyx-4, 2yx+xyx-3, xyx\right)=\\
   \left(3x^{-1}yx^2+yx^2-4, 2x^{-1}yx^2+yx^2-3, yx\right).    
\end{multline*}

Hence,
\begin{gather*}
H_o(-K_3)=3x^{-1}yx^2+yx^2-4=3yx+yx^2-4,\\ 
H_u(-K_3)=2x^{-1}yx^2+yx^2-3=2yx+yx^2-3. 
\end{gather*}

The homotopy index polynomials of the knotoids symmetric to $K_3$ are
\begin{gather*}
H_o(\bar K_3)=-3y^{-1}x^{-1}-y^{-1}x^{-1}y^{-1}+4,\quad H_u(\bar K_3)=-2y^{-1}x^{-1}-y^{-1}x^{-1}y^{-1}+3,\\
H_o(K^*_3)=-2yx-yxy+3,\quad H_u(K^*_3)=-3yx-yxy+4,\\
H_o(\rot(K_3))=3yx+yxy-4,\quad H_u(\rot(K_3))=2yx+yxy-3,
\end{gather*}

 Thus, the knotoid $K_3$ is not invertible, and the planar knotoids $K_3$, $\bar K_3$, $K^*_3$ and $\rot(K_3)$ are all different.
\end{example}

\begin{definition}
    A planar knotoid diagram $D$ is called \emph{knot-like} if the tail and the head belong to the outer region.

    A planar knotoid diagram $D$ is called \emph{proper} if the head or the tail belongs to the unbounded region. 
\end{definition}

\begin{proposition}\label{prop:knot-like_knotoid}
 1. If $K$ is a knot-like knotoid then $H_o(K)=H_u(K)=0$

2. If the tail of a knotoid $K$ belongs to the outer region then $h(K)=1$ and $H_o(K)=H_u(K)\in\Z[y,y^{-1}]$.    

3. If the head of a knotoid $K$ belongs to the outer region then $h(K)=1$ and $H_o(K)=H_u(K)\in\Z[x,x^{-1}]$.    
\end{proposition}

\begin{proof}
Let us draw cuts which connect the tail and the head with the infinity and label them by $x$ and $y$. We will calculate the homotopy index of crossings by tracking intersections with the cuts as in Example~\ref{exa:knotoid53}.

If $K$ is knot-like then the based halves do not intersect the halves, i.e. they are contractible: $h(c)=1$ for all the crossings. Hence, $H_o(K)=H_u(K)=0$.

If the tail of $K$ lies in the outer region, then the based halves of crossings do not intersect $x$-cut. Hence, for any crossing $c$, $h(c)$ is a power of $y$. Hence, $H_o(K)$ and $H_u(K)$ are polynomials in $y$. By the same reason $h(K)=y^k$ for some $k\in\Z$. But in $y^k=1$ in $\bar\pi(S,z_0,z_1)$. Then by Corollary~\ref{cor:deltaHtrivial}  $H_o(K)=H_u(K)$.

The third statement is proved analogously.
\end{proof}

\section{Homotopy polynomials for spherical knotoids}\label{sec:spherical_knotoids}

Let $S=S^2$ be the sphere. The inclusion $\R^2\hookrightarrow\R^2\cup\{\infty\}=S^2$ induces a homomorphism of the fundamental groups
\[
\xymatrix{
\pi_1(\bar\R^2,z_0) \ar[r]\ar@{=}[d] & \pi_1(\bar S^2,z_0)\ar@{=}[d]\\
\langle x,y\rangle  \ar[r] & \langle x\rangle
}
\]
given by the formula $x\mapsto x, y\mapsto x^{-1}$. The pure framed group acts trivially on $\pi_1(\bar S^2,z_0)$ and acts on $\pi_1(\bar S^2,z_0,z_1)$ by multiplication by $x$. Hence, $\bar\pi(S^2,z_0)=\pi_1(\bar S^2,z_0)=\langle x\rangle$ and $\bar\pi(S^2,z_0,z_1)=\{1\}$.

\begin{proposition}\label{prop:spherical_knotoid_symmetry}
    Let $K$ be a spherical knotoid. Then
\begin{enumerate}
    \item the homotopy type of $K$ is trivial: $h(K)=1$;
    \item $H_o(K)=H_u(K)$; we will denote the coinciding homotopy index polynomials by $H(K)$.
    \item $H(-K)=H(K)$, $H(\bar K)=-H(K)$, $H(K^*)(x)=-H(K)(x^{-1})$.
\end{enumerate}
\end{proposition}
\begin{proof}
    The second statement follows from Corollary~\ref{cor:deltaHtrivial}.

    The reflection of the sphere $\Psi$ relative to the line $z_0z_1$ induces the inversion map $\psi\colon x\mapsto x^{-1}$ on $\bar\pi(S^2,z_0)$. Hence, $H(K^*)(x)=-\psi(H(K))(x)=-H(K)(x^{-1})$.

    By Proposition~\ref{prop:reversed_knotoid}, $H(-K)=\phi(H(K))$ where the map $\phi$ is the composition of three maps: 1) inversion map $x\mapsto x^{-1}$; 2) conjugation with $[D]$ which is trivial thanks to the commutativity of $\bar\pi(S^2,z_0)$; 3) the map $\tau_*$ induced by a map $\tau$ which swaps the disks $U_0$ and $U_1$. For the sphere, $\tau_*$ is the inversion map. Hence, $\phi$ is the identity, and $H(-K)=H(K)$.
\end{proof}

\begin{proposition}\label{prop:spherical_polynomial_reduction}
    Let $K$ be a spherical knotoid $K$. Then 
\begin{enumerate}
    \item  $H(K)(x)=F_K(x,1)=F_K(1,x)$ where $F_K(x,y)$ is the two-variable polynomial defined in~\cite{KIL18};
    \item the affine index polynomial is equal to $P_K(t)=H(K)(t)+H(K)(t^{-1})$.
\end{enumerate}    
 
\end{proposition}
\begin{proof}
 1.  Let $D$ be a based diagram of $K$. For a crossing $c\in\mathcal C(D)$, $h(c)=x^k$ where $k=D^c\cdot D=i(c)$ is the intersection index~\cite{KIL18}. Hence,
\[
H_o(K)(x)=\sum_{c\colon o(c)=1}\sign(c) (h(c)-1)=\sum_{c\in\mathcal C_o(D)}\sign (x^{i(c)}-1)=F_D(x,1).
\]
Analogously, one proves that $H_u(K)(x)=F_D(1,x)$.

2. For a crossing $c\in\mathcal C(D)$, $w(c)=-o(c)i(c)$. Then
\begin{multline*}
    P_D(t)=\sum_{c\in\mathcal C(D)}\sign(c)(t^{w(c)}-1)=\sum_{c\in\mathcal C(D)}\sign(c)(t^{-o(c)i(c)}-1)=\\
    \sum_{c\colon o(c)=1}\sign (t^{-i(c)}-1)+\sum_{c\colon o(c)=-1}\sign (t^{i(c)}-1)=\\
    H_o(K)(t^{-1})+H_u(K)(t)=H(K)(t)+H(K)(t^{-1}).
\end{multline*}
\end{proof}

\begin{example}\label{exa:knotoid53_spherical}
Consider the knotoid $K_3$ in Example~\ref{exa:knotoid53} as a spherical knotoid. The substitution $y\mapsto x^{-1}$ implies that the homotopy index polynomials are 
\[
H(K_3)=x-1.
\]
Next,  $H(\bar K_3)=-x+1$, $H(K^*_3)=-x^{-1}+1$, and $\rot(K_3)=x^{-1}-1$. Thus, $K_3$, $\bar K_3$, $K^*_3$ and $\rot(K_3)$ are different spherical knotoids.

Note that the affine index polynomial has the same value on $K_3$ and $\rot(K_3)$, hence,  for spherical knotoids, homotopy index polynomial is stronger than the affine index polynomial.
\end{example}


\section*{Acknowledgments}

The author is grateful to Denis Ilyutko for fruitful discussions.


\begin{thebibliography} {100}

\bibitem{Afanasiev} D.\,M. Afanasiev, Refining virtual knot invariants by means of parity, {\em Sb. Math.}, {\bf 201}:6 (2010), pp. 3--18.

\bibitem{BG12} P. Bellingeri, S. Gervais, Surface framed braids, \emph{Geom Dedicata} {\bf 159} (2012) 51--69.

\bibitem{C} P. Cahn, A generalization of Turaev's virtual string obracket and self-intersections of virtual strings, {\em Communications in Contemporary Mathematics} {\bf 19}:4 (2017) 1650053.

\bibitem{CAL26} H. Chen, J. An, F. Li,  Coloring-allowed Invariants and the 4-phases Functions of Knotoids, arXiv:2606.10362.

\bibitem{Cheng} Z. Cheng, A polynomial invariant of virtual knots, {\em Proc. Amer. Math. Soc.} {\bf 142}:2 (2014) 713--725.

\bibitem{Cheng21} Z. Cheng, The chord index, its definitions, applications and generalizations, {\em Canad. J. Math.} {\bf 73}:3 (2021) 597--621.

\bibitem{CGX} Z. Cheng, H. Gao and M. Xu, Some remarks on the chord index,  {\em J. Knot Theory Ramifications} {\bf 29}:10 (2020) 2042003.

\bibitem{CFGMX} Z. Cheng, D.A. Fedoseev, H. Gao, V.O. Manturov, M. Xu, From chord parity to chord index, \emph{ J. Knot Theory Ramifications} {\bf 29}:13 (2020) 2043004.

\bibitem{FL24} Y. Feng and F. Li, The $F$-polynomial invariant for knotoids, J. Knot Theory Ramifications {\bf 33} (2024), no.~11, Paper No. 2450032, 22 pp.

\bibitem{FLV25} W. Feng, F. Li and A.~Y. Vesnin, A three-variable transcendental invariant of planar knotoids via Gauss diagrams, Mediterr. J. Math. {\bf 22} (2025), no.~4, Paper No. 99, 24 pp.

\bibitem{F} T. Fiedler, A small state sum for knots, \emph{Topology} {\bf 32}:2 (1993) 281--294.

\bibitem{FK} L.C. Folwaczny, L.H. Kauffman, A linking number definition of the affine index polynomial and applications, \emph{ J. Knot Theory Ramifications} {\bf 22}:12 (2013) 1341004.

\bibitem{GC26} B. Gabrov\u{s}ek, P. Cavicchioli, A table of knotoids in  up to seven crossings, arXiv:2603.06335.

\bibitem{GDS19} D. Goundaroulis, J. Dorier, A. Stasiak, A systematic classification of knotoids on the plane and on the sphere, arXiv:1902.07277.

\bibitem{GK17} N. G\"ug\"umc\"u{} and L.~H. Kauffman, New invariants of knotoids, European J. Combin. {\bf 65} (2017), 186--229.

\bibitem{Henrich}  A. Henrich, A sequence of degree one Vassiliev invariants for virtual knots, {\em J. Knot
Theory Ramifications}, {\bf 19}:4 (2010) 461--487.

\bibitem{IMN11} D.\,P. Ilyutko, V.\,O. Manturov, I.\,M. Nikonov, Parity in knot theory and graph-links, {\em M.: PFUR}, {\bf 41}:6 (2011), pp. 3--163.

\bibitem{IMN14} D.\,P. Ilyutko, V.\,O. Manturov, I.\,M. Nikonov, Virtual Knot Invariants Arising From Parities, {\em Banach Center Publ.}, {\bf 100} (2014), pp. 99--130.

\bibitem{IMN15} D.\,P. Ilyutko, V.\,O. Manturov, I.\,M. Nikonov, Parity and Patterns in Low-dimensional Topology, {\em Reviews in Mathematics and Mathematical Physics, Cambridge Scientific Publishers}, 2015.

\bibitem{ILL} Y.H. Im, K. Lee, S.Y. Lee, Index polynomial invariant of virtual links, \emph{J. Knot Theory Ramifications} {\bf 19} (2010) 709--725.

\bibitem{IPS} Y.H. Im, K.I. Park, M.H. Shin, Parities and polynomial invariants for virtual links, \emph{ J. Knot Theory Ramifications} {\bf 23}:12 (2014) 1450066.

\bibitem{J} M. J. Jeong, A zero polynomial of virtual knots, \emph{J. Knot Theory Ramifications} {\bf 25}:1 (2016) 1550078.

\bibitem{K2} L. H. Kauffman, An affine index polynomial invariant of virtual knots, {\em J. Knot Theory Ramifications} {\bf 22}:4 (2013) 1340007.

\bibitem{KPV} K. Kaur, M. Prabhakar, A. Vesnin, Two-variable polynomial invariants of virtual knots arising from flat virtual knot invariants,  {\em J. Knot Theory Ramifications} {\bf 27}:13 (2018) 1842015.

\bibitem{KIL18} S. Kim, Y.~H. Im and S. Lee, A family of polynomial invariants for knotoids, J. Knot Theory Ramifications {\bf 27} (2018), no.~11, 1843001, 15 pp.

\bibitem{KM} V.\,A. Krasnov, V.\,O. Manturov, Graph--valued invariants of virtual and classical links and minimality problem, {\em J. Knot Theory Ramifications}, {\bf 22}:12 (2013).

\bibitem{Mparity} V.\,O.~Manturov, Parity in Knot Theory, {\em  Sb. Math.} {\bf 201}:5-6 (2010) 693--733.

\bibitem{Nwp} I. Nikonov, Weak parities and functorial maps, {\em Topology}, SMFN, {\bf 51}, {\em M.: PFUR}, 2013, pp. 123--141.

\bibitem{Ntribe} I. Nikonov, Crossing tribes of tangles in a thickened surface, arXiv:2110.12446.

\bibitem{Nind} I. Nikonov, Crossing indices, traits and the principle of indistinguishability, arXiv:2110.14218.

\bibitem{Nif} I. Nikonov, Intersection formulas for parities on virtual knots,, J. Knot Theory Ramifications {\bf 32} (2023), no.~05, 2350033.

\bibitem{Next} I. Nikonov, Delta-move and the extended homotopy index polynomial, to appear in Journal of  Knot Theory  and its Ramifications.

\bibitem{ST} S. Satoh, K. Taniguchi, The writhes of a virtual knot, \emph{Fund. Math.} {\bf 225} (2014) 327--341.

\bibitem{T} V.G. Turaev, Virtual strings, \emph{Ann. Inst. Fourier (Grenoble)}, {\bf 54}:7 (2004) 2455--2525.

\bibitem{Turaev} V. Turaev, Knotoids, \emph{Osaka J. Math.} {\bf 49} (2012) 195--223.

\end{thebibliography}
\end{document}